\documentclass[conference]{IEEEtran}
\IEEEoverridecommandlockouts
\usepackage{cite}
\usepackage{amsmath,amssymb,amsfonts}
\usepackage{algorithm}
\usepackage[noend]{algpseudocode}
\algnewcommand{\Input}{\item[\textbf{Input:}]}
\algnewcommand{\Output}{\item[\textbf{Output:}]}
\usepackage{graphicx}
\usepackage{textcomp}
\usepackage{xcolor}

\definecolor{burgundy}{RGB}{128,0,32}
\definecolor{darkblue}{RGB}{0,0,139}
\definecolor{normalblue}{RGB}{0,90,180}
\usepackage[
    colorlinks=true,
    linkcolor=burgundy,   
    citecolor=darkblue,   
    urlcolor=black   
]{hyperref}
\usepackage[nameinlink,capitalize]{cleveref}
\usepackage{booktabs}
\usepackage[table]{xcolor}
\usepackage{xspace}
\usepackage{tabularx}

\newcommand{\lsr}{\langle}
\newcommand{\rsr}{\rangle}

\crefname{figure}{Figure}{Figures}
\crefname{rrule}{Reduction Rule}{Reduction Rules}

\DeclareMathOperator{\dist}{dist}
\DeclareMathOperator{\mlu}{mlu}

\newcommand{\minsr}{\textsc{Min Segment Routing}\xspace}
\newcommand{\srchallengeLong}{\textsc{Adaptive Segment Routing}}
\newcommand{\srchallenge}{\textsc{ASR}\xspace}

\usepackage{tikz}
\usepackage{tkz-berge}
\usetikzlibrary{shapes, snakes}
\usetikzlibrary{backgrounds}
\usetikzlibrary{positioning,calc,shadows,arrows.meta}

\usepackage{amsthm,float}
\usepackage{thmtools}  
\usepackage{mdframed}

\declaretheorem[name=Theorem,numberwithin=section]{theorem}
\declaretheorem[name=Proposition,sibling=theorem]{proposition}

\declaretheorem[name=Lemma,sibling=theorem]{lemma}

\declaretheorem[name=Remark,sibling=theorem,style=remark]{remark}

\newmdenv[linecolor=black,linewidth=0.8pt,roundcorner=3pt,
          innertopmargin=8pt,innerbottommargin=8pt,
          innerleftmargin=10pt,innerrightmargin=10pt]{defbox}

\newmdenv[linecolor=black,linewidth=0.8pt,roundcorner=3pt,
          innertopmargin=8pt,innerbottommargin=8pt,
          innerleftmargin=10pt,innerrightmargin=10pt]{thmbox}

\newmdenv[linecolor=black,linewidth=0.8pt,roundcorner=3pt,
          innertopmargin=8pt,innerbottommargin=8pt,
          innerleftmargin=10pt,innerrightmargin=10pt]{lembox}

\newcommand{\maxSeg}{\texttt{maxSeg}}

\newcommand{\calA}{\mathcal{A}}

\definecolor{dauphinegrey}{RGB}{65, 60, 65}
\tikzstyle{mybox} = [draw=dauphinegrey, fill=white, drop shadow, very thick, rectangle, rounded corners, inner sep=10pt, inner ysep=12pt]
\tikzstyle{fancytitle} =[fill=white, text=dauphinegrey]

\newcommand{\pboptdef}[3]{
\begin{center}
\begin{tikzpicture}
\node [mybox] (box){%
\begin{minipage}{0.44\textwidth}
    \textbf{Input: }{#2}
    \\
    \textbf{Output: }{#3}
\end{minipage}
};
\node[fancytitle, right=10pt] at (box.north west) {#1};
\end{tikzpicture}
\end{center}
}

\usepackage[colorinlistoftodos,prependcaption,textsize=scriptsize]{todonotes}

\definecolor{babypink}{rgb}{0.96,0.76,0.76}

\begin{document}

\title{On the T-Adaptive Segment Routing Problem}

\author{\IEEEauthorblockN{Amal Benhamiche}
\IEEEauthorblockA{\textit{Orange Research} \\
Chatillon, France \\
amal.benhamiche@orange.com}
\and
\IEEEauthorblockN{Kaoutar Bouaachra}
\IEEEauthorblockA{\textit{Ecole Polytechnique} \\
Palaiseau, France \\
kaoutar.bouaachra@polytechnique.edu}
\and
\IEEEauthorblockN{Yannick Carlinet}
\IEEEauthorblockA{\textit{Orange Research} \\
Chatillon, France \\
yannick.carlinet@orange.com}
\and
\IEEEauthorblockN{Morgan Chopin}
\IEEEauthorblockA{\textit{Orange Research} \\
Chatillon, France \\
morgan.chopin@orange.com}
\and
\IEEEauthorblockN{Eric Gourdin}
\IEEEauthorblockA{\textit{Orange Research} \\
Chatillon, France \\
eric.gourdin@orange.com}
\and
\IEEEauthorblockN{Nancy Perrot}
\IEEEauthorblockA{\textit{Orange Research} \\
Chatillon, France \\
nancy.perrot@orange.com}
}

\maketitle

\begin{abstract}
In this paper, we present a multi-period optimization problem arising from the traffic engineering of core IP/MPLS  networks, called $T$-\srchallengeLong. This problem aims at computing a sequence of segment routing paths that adapt to a multi-period scheduled maintenance, while allowing limited number of path reconfigurations between successive time steps.
Rather than focusing solely on minimizing the classical Maximum Link Utilization (MLU) criteria, we introduce a more refined lexicographic objective that minimizes the entire sorted vector of link loads providing a more efficient resource utilization for all the links.
We propose a generic approach that decomposes this problem into a sequence of subproblems. To model each subproblem, we introduce two Mixed Integer Linear Programming (MILP) formulations, ALEXA and STELA, resulting in two variants of our approach. Finally, we evaluate and compare the efficiency of both variants on small to medium-sized network instances.

\end{abstract}


\section{Introduction} 
\label{s:intro}
With the arrival of next-generation networks enabled by disruptive technologies such as Software-Defined Networking (SDN) and Network Virtualization, network operators have to reconsider \textit{Traffic Engineering} (TE) challenges and conventional solutions through a different lens. These challenges include the adaptation of networks to accommodate the ever-growing volume of traffic and the multiplicity of user services and contents, all while maintaining high quality of service (QoS). The objective is to rely on TE to better tune network parameters so as to make a more efficient use of the existing infrastructure and resources in order to avoid congestion and hence to improve the QoS experienced by users of the multiple services supported by the network. A standard design rule requires the network to ensure service continuity regardeless of any link or node failure that might occur. The term failure should be understood broadly, covering cases where a link becomes unavailable either due to unforeseen events or deliberate (and scheduled) interventions by a network operator. Moreover, today’s geopolitical tensions and accelerating climate change increase the need for networks that are not only secure but also resilient to multiple router or link failures caused by natural disasters or acts of sabotage.

In this context, designing an algorithm that, given an IP/MPLS network and a traffic matrix (that is, a set of projected values based on current measured volumes of traffic between given origin/destination pairs), computes a routing scheme that optimizes network load, even in scenarios of network equipment unavailability, becomes a crucial task.

The proposed routing scheme relies on the so-called Segment Routing (SR) protocol, a network technology proposed by IETF~\cite{rfc8402} which recently attracted much attention in both telecommunication networking and Operations Research communities. Indeed, the packets in an IP/MPLS network are traditionally routed along the shortest paths from the origin to the destination according to a set of arc weights set by the network administrators. While this approach is easy to implement in practice, it comes with several limitations such as the complexity of finding a set of link weights that induce shortest-paths with the least possible network congestion while remaining robust against scenarios where multiple network components become unavailable.

Segment Routing was designed to alleviate this issue by enabling the possibility to route packets over non-shortest paths without extensive modifications to the network. More precisely, each packet entering the network is assigned to a so-called \emph{segment path}, which is a sequence of routers referred to as \emph{node segments} or \emph{waypoints} that the packet must visit, one after the other in the network before reaching its final destination. Between two waypoints, traditional shortest path-based routing is used. By encoding routing instructions directly into the packet header, SR allows packets to deviate from the single shortest paths from origin to destination, and hence allows for greater flexibility with minimal additional implementation costs for network administrators.

The design of sets of segment paths that account for multiple routers or links unavailability is a very challenging yet crucial optimization problem. Addressing this is critical to ensure \emph{stability} and enable deployment in modern, large-scale IP/MPLS networks. In real IP/MPLS networks, routing
changes cannot be applied arbitrarily from one period to the next: they require
validation, deployment, supervision, and sometimes manual intervention.
Moreover, scheduled maintenance operations must be anticipated without
degrading service quality. The challenge therefore stands not only in computing
"good enough" routing schemes for each network state, but in producing a sequence of
routing schemes that remains close enough to be operationally deployable.

This paper addresses the $T$-\srchallengeLong ($T$-\srchallenge) problem. This problem consists, given an IP/MPLS network, a set of traffic demands and multiple time-periods, each associated with scheduled intervention scenarios, in computing a sequence of segment routing paths under budget and length (in terms of number of segments) requirements. 

In the literature of TE and IP/MPLS network optimization, the commonly used objective function is the minimization of the Maximum Link Utilization (MLU). While this approach provides a partial optimization of network resources, it has limitations. Specifically, two routing schemes with identical MLU values may differ significantly in their residual capacities on other links, or even cause some links to become more heavily loaded while efforts are made to reduce the most congested link. Therefore, and to address these issues, our main contribution is the introduction of a lexicographic objective function based on the sorted vector of link loads across all time periods. This approach first minimizes the largest link load, then the second-largest, and so forth. Such a criterion offers a more refined optimization compared to simple MLU minimization. Unlike average loads or weighted sums, it requires no tuning parameters, avoids hiding local congestion behind many lightly loaded links, and directly aligns with operational priorities by focusing on eliminating the most critical bottlenecks first.

Ultimately, a solution to this problem yields a resilient or \textit{adaptive} segment routing scheme that adheres to budget and length constraints while lexicographically minimizing link loads.

The remainder of the paper is organized as follows. In Section \ref{s:sota}, we give an overview of the related works in the literature of TE and IP/MPLS network optimization. We give some definitions and notations in Section \ref{sec:notations} and we use them to introduce the $T$-\srchallenge problem more formally and model it as a MILP formulation in Section \ref{sec:prb def}. We describe our exact approaches in Section \ref{exact} and assess their efficiency through a set of preliminary results in Section \ref{sec:experiments}. Finally, some concluding remarks are given in Section \ref{sec:conclusion}.

\section{State of the Art}
\label{s:sota}
From the complexity point of view, the $T$-\srchallenge problem studied in this paper generalizes the \minsr problem~\cite{hartert2015solving}:
indeed, one can easily verify that \minsr is equivalent to~$\{0\}$-\srchallengeLong{}.
Consequently, the computational intractability results established for \minsr
carry over directly to $T$-\srchallenge{}.
In particular, this problem remains parameterized intractable and inapproximable
(under standard complexity assumptions), even on highly restricted network
topologies~\cite{camille25}.

Most existing optimization approaches focus on the static segment routing
problem, or on variants of it that do not jointly consider multi-period
planning, scheduled link interventions, and bounded waypoint reconfigurations. Bhatia et al.~\cite{bhatia} formulate SR with two
segments as a linear program using a path-based model.
Building on this, Hartert et al.~\cite{hartert2015solving} quantify the
benefit of SR over standard routing protocols, demonstrating congestion
reductions of 30\% to 50\% under known traffic, and introduce a
MILP-based heuristic to reduce computation times.
Schüller et al.~\cite{schuller} further extend this line of work with a
failure-resilient MILP that minimizes the Maximum Link Utilization (MLU),
though at the cost of prohibitive computational times on large instances.
A model supporting an arbitrary number of segments was proposed by Jadin
et al.~\cite{jadin}, who combine column generation with a branch-and-bound
procedure to tackle the problem at scale. Pereira et al.~\cite{pereira} propose a heuristic for
three-segment SR that incorporates link failure handling and makes use of
adjacency segments.
Gay et al.~\cite{gay_srls} introduce SRLS, a local search heuristic aimed
at fast re-optimization in response to sudden traffic changes, addressing
the practical limitation of MILP-based approaches in time-sensitive
scenarios.
Parham et al.~\cite{parham} consider the joint optimization of SR paths
and shortest-path routing weights, proposing a heuristic for this richer
combined problem. These works are the closest algorithmic references for our study.
However, they remain essentially static: the routing configuration is optimized
for a single traffic and network state. They therefore do not cover the multi-period setting with reconfiguration constraints considered here. In contrast, the problem considered in this paper is inherently multi-period:
routing decisions must be planned over a horizon with scheduled interventions,
while limiting the number of waypoint changes between consecutive time steps. Moreover, rather than minimizing the MLU alone, we adopt a lexicographic
objective over the sorted vector of arc loads, as formally defined in~\cref{problem definition}.
This objective strictly strengthens MLU minimization: any lex-optimal solution
is also MLU-optimal, but the converse does not necessarily hold. Beyond this algorithmic gap, we also propose a set of benchmark instances
specifically designed for this problem, including network topologies, traffic
matrices, intervention scenarios, and budget parameters.
These instances are made publicly available to the community, providing a
common basis for testing future approaches on this problem
setting.

Unlike the single-period SR optimization studied by Callebaut et al.~\cite{callebaut},
dominated segment paths; in particular paths containing loops; cannot be
removed from the search space in T-ASR without losing optimality. A path that is dominated in one isolated period may be necessary because it remains close to the previous routing scheme and therefore satisfies the budget (see~\cref{fig:interventionsar3}). This shows that the budget constraint  can force an optimal solution to include segment paths containing loops, which the preprocessing of~\cite{callebaut} would discard.

\begin{figure}[h]
    \centering
\includegraphics[width=0.6\linewidth]{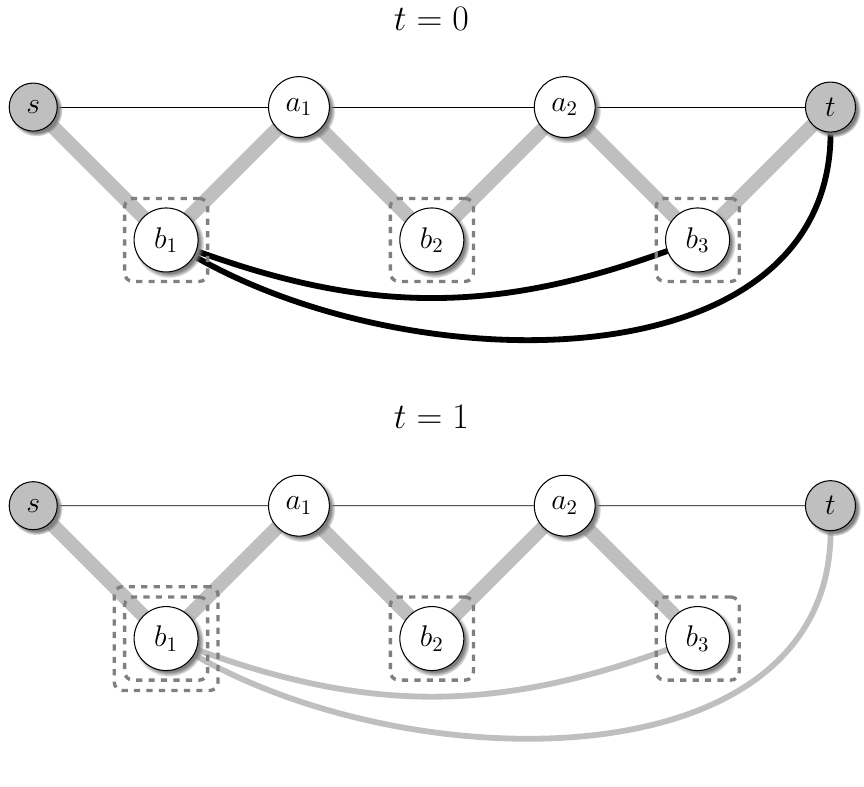}
    \caption{Example of an instance of $\{0,1\}$-\srchallenge where the only optimal solution includes a segment path containing a loop. The input graph is bidirected with each arc~$uv$ having the same capacity and weight as its reverse~$vu$. For this reason, we represent~$uv$ and $vu$ as a single edge in the figure for clarity. The thickness of an edge represents the arcs capacity value. Thin, medium and thick edges correspond to capacity~$1$,~$2$ and~$3$, respectively. All arcs have unit weights except for the arcs $(b_1,t), (t,b_1)$ which have weight~$2$. The only demand to be routed has traffic volume~$1$. The budget is $\kappa(1) = 3$ and the intervention scenario is~$q(1) = \{(b_3,t), (t,b_3)\}$. The ``grayish'' links correspond to the arcs traversed by the flow. An optimal segment paths solution is~$p_1 = \lsr b_1,b_2,b_3 \rsr$ and~$p_2 = \lsr b_1,b_2,b_3,b_1 \rsr$ at time~$0$ and~$1$ with induced MLUs~$\frac{1}{3}$ and~$\frac{1}{2}$, respectively. Notice that~$p_1$ yields the best MLU at time~$0$ since it only uses arcs with maximum capacity and all three waypoints are necessary to achieve this. At time~$t=1$, the given budget allows only the addition or deletion of a single waypoint. Deleting any of~$p_1$'s waypoints or adding no waypoint at all would divert the flow on a thin arc, inducing a MLU of~$1$. The remaining possibility without creating a loop is adding a waypoint on $a_1$ or $a_2$. However, each of these possibilities would again divert the flow to at least one thin arc. Hence, the only remaining option is to add~$b_1$ as a new waypoint to get a MLU of~$\frac{1}{2}$.}
    \label{fig:interventionsar3}
\end{figure}

\section{Definitions} \label{sec:notations}

In simple terms, the problem consists of determining the segment path for each demand in order to balance the network load.
Prior to providing the formal definition of the optimization problem, we introduce the necessary preliminary notations and definitions.


\noindent\textbf{Network} 
A \emph{network} is a tuple $(G=(V,A), \omega, c)$ where $G$ is a directed graph of~$n$ vertices with vertex set~$V = \{v_0, v_1, \ldots, v_{n-1} \}$ and arc set $A \subseteq \{ a_{i,j} : (v_i, v_j) \in V \times V\}$, $\omega: A \rightarrow \mathbb{Q}$ is a weight function used to compute the shortest paths in $G$, and $c\colon A \rightarrow \mathbb{Q}$ is a capacity function that represents, for each arc $a \in A$, its bandwidth $c(a)$ (given in Mbits or Gbits in practice), that is, the maximum amount of traffic throughput that it can accommodate. The terms ``arc'' and ``link'' are used interchangeably throughout this document.

\noindent\textbf{Forwarding graph \& ECMP} \label{forwarding_graph}
Let $(G=(V,A), \omega, c)$ be a network. The \emph{forwarding graph} from a \emph{source} $u \in V$ to a \emph{target} $v \in V$ is the subgraph of $G$ containing all arcs that belong to any shortest path (according to weights $\omega$) from $u$ to $v$. It is denoted $FG(u,v)$. When $FG(u,v)$ is not a simple path, the \emph{Equal-Cost Multi-Path} (ECMP) mechanism is activated: the incoming flow at a vertex of $FG(u,v)$ is evenly divided between the outgoing arcs of this vertex in $FG(u,v)$ (see~\cref{fig:definitions}). 

\noindent\textbf{Segment routing path} 
A \emph{segment routing path}, or \emph{segment path}, in a network $(G=(V,A), \omega, c)$ is a succession of forwarding graphs such that the target of the previous forwarding graph coincides with the source of the next one. A segment path is denoted by $\lsr s, w_1, \ldots, w_\ell, t \rsr$ where $s$ is the source of the segment path, $t$ the target, and $w_1, \ldots, w_\ell \in V$ are the \textit{waypoints} in that order. The source and target do not count as waypoints. The case where there is no waypoint (i.e. $l = 0$) is possible; it means that the demand is simply routed on the shortest paths from~$s$ to~$t$.

\noindent\textbf{Segment} 
In this context, we define a \textit{segment} as a pair of successive nodes in a segment path. These nodes are either the source, target or waypoints included in a segment path. For instance, ($s$, $w_{1}$), ($w_{1}$, $w_{2}$), \ldots, ($w_\ell$,~$t$) are the segments composing the segment path $\lsr s, w_1, \ldots, w_\ell, t \rsr$. Each segment~$(u,v)$ of a segment path is thus associated with the forwarding graph~$FG(u,v)$. A typical maximum number of segments in a segment path is between 4 and 10, depending on the underlying protocol and the router technology.

\noindent\textbf{Demand} 
A \emph{demand} on a network~$(G=(V,A), \omega, c)$ is a pair~$(s,t)$ where the \emph{terminals} $s,t \in V$ are respectively the \emph{source} and \emph{target} of the demand. Each demand is associated with a traffic volume, i.e. the size of the flow that needs to be routed from the source to the target. If the traffic volume is 1, it is referred to as a unit demand. The volume can vary in time, therefore the traffic volume is a function denoted $\nu:  D \times T \to \mathbb{Q}$ with $T$ the discrete set of time periods and $D$ the set of demands, referred to as the traffic matrix.

\noindent\textbf{Routing scheme} 
Given a network $(G=(V,A), \omega, c)$ and a set of $k$ demands $D = \lbrace (s_i, t_i) : i=1, \ldots, m \rbrace$ on $G$, a \emph{routing scheme} for $D$ is a set~$P$ of $k$ segment paths $\lbrace p_d : d \in D \rbrace$ such that $p_d$ is a segment path from~$s$ to~$t$ associated to the demand~$d=(s,t)$.

\noindent\textbf{Time horizon} 
The routing scheme has to be decided in advance and over a given time period (from a few days to a couple of weeks). We assume that the traffic demand values can be estimated through forecasting for each time period, whereas the network state (resource status, utilization and availability) depends on the scheduled maintenance interventions. We denote by $T$ = $\{0, 1, \ldots, h-1\}$ the set of time periods, and~${h \geq 1}$ the routing planning horizon. For example, $T$ = $\{0, 1, \ldots, 6\}$ for a daily planning over a week and $T$ = $\{0, 1, 2, 3\}$ for a weekly planning over a month. The time step $0$ is a special case because it represents the nominal situation, when the network runs without any failure. Therefore, we also denote $T^* = T \setminus \{0\}$.

\noindent\textbf{Interventions} 
An intervention in the network (for maintenance, or other reason) causes a link or a node to be turned down for a certain time. Functionally, it is the same as a failure, except it can be planned in advance. In the following, we consider only links turned down, because if a node is down, it is functionally the same as if all connected links to this node are down. 

\noindent\textbf{Intervention scenario}
An \emph{intervention scenario} is a function~$q: T^* \to 2^A$ that, given a time step~$t \in T^*$, returns the subset~$q(t) \subseteq A$ of arcs in~$G$ that are affected by a scheduled maintenance operation through one or several such \textit{interventions}. We therefore denote by~$G_t$ the subgraph of~$G$ with set~$A \setminus q(t)$ of available arcs, \textit{i.e.} taking into account the intervention scenario that takes down some links. 

\noindent\textbf{Budget}
For a given network and a set of demands, the routing scheme may require some changes or reconfiguration from one time period to the next, in order to take into account the links and nodes that are down. Each change in a routing scheme may require a human action that has a cost. In this context, it is desirable to limit the number of network reconfigurations (i.e. adding/removing waypoints). We denote by~$\dist(P, P') \in \mathbb{N}$ the number of changes between two routing schemes~$P$ and $P'$. More formally, 

$$\dist(P, P') = \sum_{d \in D, i,j \in V, i\neq j} |\delta(p_d,ij) - \delta(p_{d}',ij)|$$
where~$\delta(p_d,ij)$ is equal to 1 if the segment path~$p_d$ associated to the demand~$d$ contains the segment $(i, j)$,~$0$ otherwise (see~\cref{tab:distances}).

\begin{table}[h!]
\centering
\rowcolors{3}{gray!10}{white}
\setlength{\tabcolsep}{5pt}
\renewcommand{\arraystretch}{1.1}
\begin{tabular}{l|l|l}
\rowcolor{gray!25}
\toprule
\textbf{Routing Scheme} $P$ & \textbf{Routing Scheme} $P'$ & $\dist(P, P')$ \\
\midrule
$\{\lsr s, w,  t \rsr\}$ & $\{\lsr s, t \rsr\}$   & $3$ \\
$\{\lsr s, w_1, w_2,  t \rsr\}$ & $\{\lsr s, w_1, t \rsr\}$   & $3$ \\
$\{\lsr s, w_1, w_2,  t \rsr\}$ & $\{\lsr s, w_2, w_1, t \rsr\}$   & $6$ \\
$\{\lsr s, w_1,  t  \rsr\}$ & $\{\lsr s, w_2,  t \rsr\}$   & $4$ \\
$\{\lsr s_1, w_1,  t_1  \rsr, $ & $\{\lsr s_1, w_4,  t_1 \rsr, $   &  \\
\rowcolor{white}
$\lsr s_2, w_2, w_3, t_2 \rsr\}$ & $\lsr s_2, w_3, w_2, t_2 \rsr\}$   & $10~(= 4 + 6)$ \\
\end{tabular}
\vspace{0.1cm}
\caption{Examples of distance values}
\label{tab:distances}
\end{table}

For each time step, a budget function is given, denoted $\kappa : T^* \to \mathbb{N}$, that represents the maximum number of changes allowed from one time step to next one.

\paragraph{Load} 
The load $\lambda(a, t) \in \mathbb{Q}^{+}$ of arc $a \in A$ at time step $t \in T$, is the ratio between the quantity of flow using an arc and its capacity. This ratio is also referred to as \textit{link utilization} and is often used as (part of) the optimization criterion in the literature related to Traffic Engineering. Note that it can be higher than 1 in case of congestion on the link. 
In order to compute the load, we introduce the concept of \textit{split coefficients}, denoted by $r(u, v, a, t)$. They represent the proportion of flow between source node $u$ and target node $v$ that passes through arc $a$ of $FG(u, v)$ under the state of the network at time $t$ (i.e. in graph $G_t$). 
The ratios are given for all couples of nodes $(u, v)$ because the segment $(u, v)$ can be potentially used for routing a demand. 
Note that $r(u, v, a, t)$ is always greater or equal to $0$, and $r(u, v, a, t) = 0$ if arc $a$ is not in the forwarding graph $FG(u, v)$. The load of an arc $a \in A$ at time step $t \in T$ is then:

$$
\lambda(a, t) = \frac{\sum_{d \in D}\sum_{i,j \in V} r(i, j, a, t) \; \nu(d,t) \; \delta(p_d,ij)}{c(a)}
$$
See \cref{fig:definitions} for an illustration of ECMP, split coefficients and of a segment path with one waypoint.

\begin{figure}[h!]
    \centering
    \includegraphics[width=0.8\linewidth]{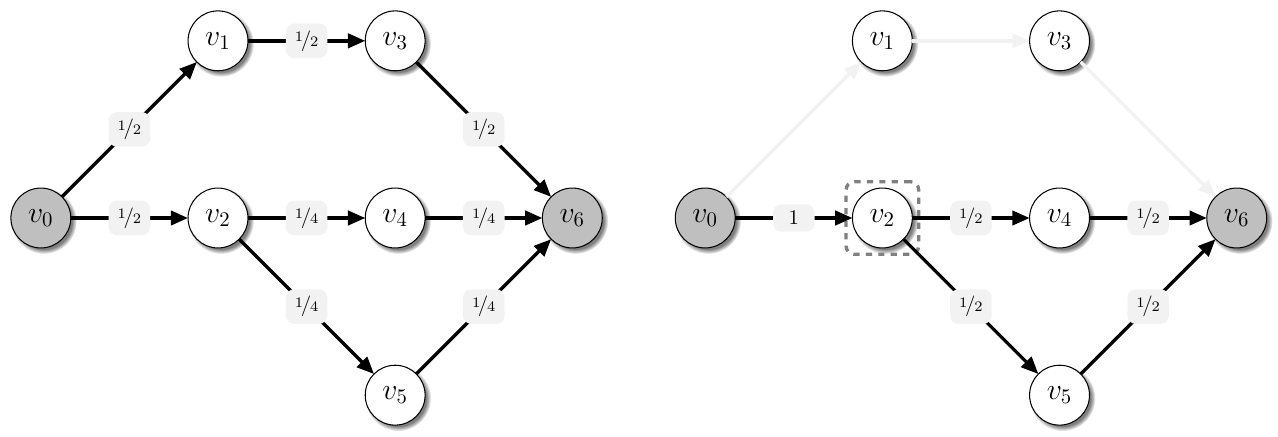}
    \caption{On the left, the demand~$(v_0, v_6)$ is routed on a network with unit weights through the segment path~$\lsr v_0, v_6 \rsr$. The associated forwarding graph, $FG(v_0, v_6)$, includes all arcs. The fraction on each arc indicates how the flow is split among the shortest paths (split coefficient) using the ECMP rule. On the right, a waypoint in~$v_2$ is introduced, resulting in the flow being routed through the segment path~$\lsr v_0, v_2, v_6 \rsr$ that is, through the bottom part of the network. The associated forwarding graphs, $FG(v_0, v_2)$ and $FG(v_2, v_6)$, are represented with bold arcs.}
    \label{fig:definitions}
\end{figure}

\paragraph{Maximum Link Utilization (MLU)} 
Given a network $(G=(V,A), \omega, c)$, a set of $k$ demands $D = \lbrace (s_i, t_i) : i=1, \ldots, k \rbrace$ on $G$, a \emph{routing scheme}~$P$ for $D$, the maximum link utilization, denoted~$\mlu(P,D,G)~\in~\mathbb{Q}^{+}$ is the load of the most loaded link.

\section{Problem Formulation}
\label{sec:prb def}

In this section, we introduce the problem called $T$-\srchallengeLong~($T$-\srchallenge{}). A formal expression of the decision variables, constraints and objective function of $T$-\srchallenge{} is provided

\subsection{Decisions}

\paragraph{Routing}
Let us define the binary variable $x^{dt}_{ij}$ that takes value 1 if the segment path of demand $d\in D$ contains segment ($i$, $j$) at time step $t\in T$ and 0 otherwise. 

\noindent Therefore, a solution of the $T$-\srchallenge{}  problem corresponds to a \textit{routing scheme} induced by the variables $x$ with non-zero values, that minimizes the objective described in~\cref{ss:obj}, and satisfies the constraints given in~\cref{ss:constraints}, for each period of time. 

\subsection{Constraints} \label{ss:constraints}

\paragraph{Flow-conservation constraints} 
The following set of equalities ensures that each traffic demand $d\in D$ is routed along one segment path which connects its origin and its destination, at each time step $t$, \textit{i.e.}, for all~$i \in V$, $d =(s,t) \in D$ and $t\in T$
\begin{align}
\label{ctn:flow} & \sum_{j \in V\setminus \{i\}} x^{dt}_{ij} - \sum_{j \in V\setminus \{i\}} x^{dt}_{ji} =  
\left \{ \begin{array}{ll}
1 & \mbox{if $i = s$,}\\
-1 & \mbox{if $i=t$,}\\
0 & \mbox{otherwise.}
\end{array} \right.
\end{align}

\paragraph{Number of segments} 
The following inequalities allow, for each demand $d$ and each time step $t$, to limit the number of segments, including the first hop from the source node of $d$ to the first waypoint visited, and the last waypoint visited to the destination node, used in the solution. Formally, for all~$d\in D$ and~$t\in T$,
\begin{align}
\label{ctn:segments} & \sum_{i,j \in V} x^{dt}_{ij} \leq \texttt{maxSeg}
\end{align}

\paragraph{Load constraints}

The total traffic of a link is computed by adding the traffic of all demands that are routed through this link.
Formally, for all~$t \in T$ and $a\in A \setminus q(t)$,
\begin{align}
\label{ctn:load} & \sum_{d \in D}\sum_{i,j \in V} r(i, j, a, t) \; \nu(d,t) \; x^{dt}_{ij}  \leq \lambda(a, t) \; c(a)
\end{align}

\paragraph{Budget constraints} 
Since each network configuration change incurs significant operational costs and may also risk deteriorating the Quality of Service, it is desirable to incorporate budget limitations in the model. The following set of inequalities (\ref{ctn:budget_3}) ensures that the value returned by~$\dist$ (defined in~\cref{sec:notations}, paragraph {\it Budget}) is indeed bounded by $\kappa$ and allows to restrict the number of waypoint changes scheduled from one time step to the next.
Formally, for all~$t \in T^*$,
\begin{align}
\label{ctn:budget_3} & \sum_{d\in D}\sum_{i,j\in V,i\neq j}  | x^{dt}_{ij} - x^{d, t-1}_{ij} | \leq \kappa(t)
\end{align}

\subsection{Objective} \label{ss:obj}
A \emph{routing scheme} $P=\{p_d : d\in D\}$ assigns to every demand a
segment path at each time step; equivalently, $P$ is described by the binary
variables $x^{dt}_{ij}$, and we write $P\in\mathcal{F}$ when $P$ satisfies the
constraints~\eqref{ctn:flow}--\eqref{ctn:budget_3}. Let
$\calA_T = \{(a,t): t\in T,\ a\in A\setminus q(t)\}$ be the set of available arc--time pairs (with $q(0)=\varnothing$), and $N=|\calA_T|$. Each routing scheme $P$ induces a load $\lambda(a,t;P)$ on every pair $(a,t)\in\calA_T$, and we denote by
\begin{equation*}
  \label{eq:sorted-vector}
  \lambda^{\downarrow}(P)=\bigl(\lambda^{\downarrow}_{1}(P),\,\ldots,\,
  \lambda^{\downarrow}_{N}(P)\bigr),\qquad
  \lambda^{\downarrow}_{1}(P)\geq\cdots\geq\lambda^{\downarrow}_{N}(P),
\end{equation*}
the vector of these loads sorted in non-increasing order, so that
$\lambda^{\downarrow}_{k}(P)$ is the $k$-th largest load. Crucially, this order
is \emph{not} fixed a priori: the permutation that sorts the loads---and hence
the arc--time pair realizing the $k$-th largest load---depends on the routing
scheme $P$ itself. The objective is to minimize this sorted vector
lexicographically,
\begin{equation}
  \label{ctn:objlex}
  \operatorname*{lex\,min}_{P\in\mathcal{F}}\ \lambda^{\downarrow}(P),
\end{equation}
that is, $\lambda^{\downarrow}_{1}(P)$ is minimized first (the maximum link
utilization), then $\lambda^{\downarrow}_{2}(P)$ among all schemes achieving the
optimal first component, and so on.

\subsection{Problem definition} \label{problem definition}

We are now in position to give the formal definition of the $T$-\srchallenge{} problem formally.

\noindent
\pboptdef{$T$-\srchallengeLong~($T$-\srchallenge{})}
{A network~$(G=(V,A), \omega, c)$, a set of demands~$D = \lbrace (s_i, t_i) : i=1, \ldots, k \rbrace$, a traffic volume $\nu: D \times T \to \mathbb{Q}$, set of time periods~$T = \{0, 1, \ldots, h-1\}$ with horizon~$h\geq1$, an intervention scenario~$q: T^* \to 2^A$, a budget~$\kappa : T^* \to \mathbb{N}$ and a maximum number of segments~$\maxSeg \in \mathbb{N}$ ($\maxSeg \geq 2$).}
{A routing scheme~$P\in\mathcal{F}$ (one segment path per demand and time step satisfying constraints~\eqref{ctn:flow}--\eqref{ctn:budget_3}) that minimizes the objective~\eqref{ctn:objlex}.}

\section{Exact Resolution}
\label{exact}

The lexicographic objective~\eqref{ctn:objlex} cannot be expressed as a single
linear objective over $\mathcal{F}$. It minimizes the sorted load vector, whose ordering depends on the routing solution itself.
A weighted sum of the sorted loads $\sum_{k} w_k\,\lambda^{\downarrow}_k(P)$, also known as Ordered Weighted Averaging (OWA)\cite{OGRYCZAK200380}, may
seem a natural alternative, but it is unsuitable here. Recovering~\eqref{ctn:objlex} requires weights separated enough that no gain on lower ranks
can compensate a loss on a higher one, which forces ratios~$w_k/w_{k+1}$ growing
with the number of ranks and quickly exceeds precision range. 
We therefore manage to decompose the problem into a sequence of MILPs solved iteratively. At each iteration, one rank of the sorted load vector is optimized while the previously optimized ranks are fixed. In what follows, we propose two exact methods that instantiate this generic framework. First, \emph{ALEXA} (described in \cref{subsec:alexa}) explicitly models the complete ordering through an assignment formulation, and second, \emph{STELA}
(\cref{subsec:stela}) that identifies one rank at a time using threshold-exceedance
constraints.
A comparison of both methods is discussed in
\cref{subsec:results-depth}.

\subsection{Overview and common Framework}\label{subsec:exact-overview}

Let $\mathcal F$ denote the set of routing schemes satisfying
Constraints~\eqref{ctn:flow}--\eqref{ctn:budget_3}.
Notice that~$\mathcal F \neq \varnothing$ since the routing scheme that consists in assigning the segment path $\lsr s, t \rsr$ to each demand~$(s,t) \in D$ at every time step trivially satisfies~\eqref{ctn:flow}--\eqref{ctn:budget_3}.

Each routing scheme $P \in \mathcal{F}$ consists of finitely many segment
paths, each a sequence of at most $\texttt{maxSeg}$ pairwise distinct
segments taken from the finite set $\{(i,j) : i,j \in V,\, i \neq j\}$;
hence $\mathcal{F}$ is finite, and it is non-empty since the segment-free
scheme is feasible. A routing scheme $P \in \mathcal{F}$ induces, for every
demand $d \in D$, time step $t \in T$, and distinct nodes $i,j \in V$, a
value $x^{dt}_{ij}(P) \in \{0,1\}$ equal to $1$ when demand $d$ is routed
through segment $(i,j)$ at time $t$ and $0$ otherwise. We write
$x(P) \in \{0,1\}^m$ for the vector collecting these values under a fixed
indexing of the coordinates over all $(d,t,i,j)$. Making the dependence on $P$ explicit through
$x^{dt}_{ij}$, the load $\lambda(a,t)$ of an available arc--time pair
$(a,t)\in\mathcal{A}_T$ introduced in~\cref{sec:notations} reads
\begin{equation}
\lambda(a,t;P)=\frac{1}{c(a)}
\sum_{d\in D}\sum_{i,j\in V}
r(i,j,a,t)\,\nu(d,t)\,x^{dt}_{ij}, \label{eq:load-P}
\end{equation}
and we use $\lambda(a,t)$ and $\lambda(a,t;P)$ interchangeably, keeping the
explicit form wherever the dependence on $P$ matters.


Recall from~\cref{ss:obj} that
$\mathcal{A}_T=\{(a,t):t\in T,\;a\in A\setminus q(t)\}$ (with
$q(0)=\varnothing$) and $N=|\mathcal{A}_T|$.

We address the lexicographic optimization problem~\eqref{ctn:objlex} by decomposing it into a sequence of single-objective optimization problems, each optimizing a specific rank of the sorted load vector, instead of solving it in a single step. We
set $\mathcal{F}_0=\mathcal{F}$ and define  for each iteration $k\ge1$ $(\bar{\lambda}_k,\mathcal{F}_k) $ such that:
\begin{equation}
\small
\bar{\lambda}_k=\min_{P\in\mathcal{F}_{k-1}}\lambda^{\downarrow}_{k}(P)
\label{eq:lex-recursion1}
\end{equation}

\begin{equation}
\small
\mathcal{F}_k=\{P\in\mathcal{F}_{k-1}:\lambda^{\downarrow}_{k}(P)=\bar{\lambda}_k\}.
\label{eq:lex-recursion2}
\end{equation}

This kind of iterative frameworks is well-known in the literature of lexicographic optimization.
In particular, Abernethy et al.~\cite{abernethy2024} analyze the generic
lexicographic \emph{maximization} of an arbitrary set $X\subseteq\mathbb{R}^n$,
where iteration $k$ maximizes the $k$-th smallest component while fixing
the previously optimized ones; their result establishes that the procedure returns exactly the lexicographic optima of $X$.
In this work, we consider the order-reversed (leximin) counterpart that is to say we \emph{minimize} the
$k$-th largest load $\lambda^{\downarrow}_{k}$ instead of maximizing the $k$-th
smallest component, which is equivalent up to negation of the objective function.
Since $\mathcal{F}\subseteq\{0,1\}^m$ is finite and non-empty, a lexicographic
minimizer exists, the minima in~\eqref{eq:lex-recursion1} are attained, and a
scheme $P^\star\in\mathcal{F}$ is lexicographically optimal for
$T$-\srchallenge{} if and only if $P^\star\in\mathcal{F}_N$.

The expression~\eqref{eq:lex-recursion1} defines the values $\bar{\lambda}_k$ but
not how to compute them. We call $\mathcal{M}_k$ a mixed-integer linear formulation corresponding to
step~$k$ of~\eqref{eq:lex-recursion1}, and we say it is \emph{valid} if its optimal
value equals $\bar{\lambda}_k$ (as defined in~\eqref{eq:lex-recursion1}--\eqref{eq:lex-recursion2}) and every
optimal solution yields a scheme in $\mathcal{F}_k$.

Then $\textbf{Solve}(\mathcal{M}_k)$ denotes a
single call to a MILP solver on that formulation, returning the pair
$(\bar{\lambda}_k,P^\star_k)$, where $P^\star_k$ is the  

\begin{algorithm}[ht]
\algrenewcommand\algorithmicindent{0.6em}
\begin{algorithmic}[1] 
\small
\caption{Generic lexicographic minimization}
\label{alg:lex-template}
\Input{$G$, $D$, $T$, $q$, $\kappa$, \maxSeg; a family $\{\mathcal{M}_k\}$}
\Output{Lex-optimal routing $P^\star$ and
$(\bar{\lambda}_1,\ldots,\bar{\lambda}_N)$}

\For{$k\gets1$ To $N$}
    \State $(\bar{\lambda}_k, P^\star_{k}) \leftarrow$\textbf{Solve}($\mathcal{M}_k$) \Comment{Minimize $\lambda_k^{\downarrow}(P)$ over $\mathcal{F}_{k-1}$}


    \If{$\bar{\lambda}_k=0$}
        \State \Return $P_k^\star,\ (\bar{\lambda}_1,\ldots,\bar{\lambda}_k,{0,\ldots,0})$ \Comment{$N-k$ zeros}
    \EndIf
\EndFor

\State  \Return $P_N^\star,\ (\bar{\lambda}_1,\ldots,\bar{\lambda}_N)$\;
\end{algorithmic}
\end{algorithm}

\cref{alg:lex-template} presents the generic sequential lexicographic
procedure that we propose for $T$-\srchallenge{}. At iteration $k$, and given a MILP formulation $\mathcal{M}_k$, it returns the optimal solution $\bar{\lambda}_k=\min_{P\in\mathcal{F}_{k-1}}\lambda^{\downarrow}_{k}(P)$
along with a minimizer. 

\cref{alg:lex-template} instantiates this scheme and
returns a lexicographically optimal routing scheme after at most $N$ MILP
solves~\cite{abernethy2024}. In what follows, we show that the validity of $\mathcal{M}_k$, i.e.,
that it actually optimizes $\lambda^{\downarrow}_{k}$ over $\mathcal{F}_{k-1}$, is
ensured for each formulation according to lemmas of
\cref{subsec:alexa,subsec:stela}. In particular, we instantiate $\mathcal{M}_k$ via two
alternative formulations, denoted $\mathcal{M}_k^{\mathrm{ALEXA}}$ and
$\mathcal{M}_k^{\mathrm{STELA}}$, and detailed hereafter.

\subsection{ALEXA: Assignment-based Lexicographic Algorithm}
\label{subsec:alexa}

ALEXA encodes the full ordering in a single MILP. Let~$\mathcal{K}=\{1,\ldots,N\}$
denote the set of ranks. We introduce a permutation variable
$\pi_{a,t,r}\in\{0,1\}$, equal to $1$ iff the load~$\lambda(a,t;P)$ has rank $r \in \mathcal{K}$,
and rank-load variables~$L_r\in\mathbb{R}_{+}$ giving the $r$-th largest load.
The model $\mathcal{M}_k^{\mathrm{ALEXA}}$, $k\in\mathcal{K}$, can therefore be expressed as follows:

\begin{subequations}
\label{mod:alexa}
\small
\begin{align}
\min_{P\in\mathcal{F},\,\pi,\,L}\ & L_k
  \label{alexa:obj}\\
\text{s.t.}\ & \eqref{ctn:flow}\text{--}\eqref{ctn:budget_3},
  \label{alexa:feas} \\
& \sum_{r\in\mathcal{K}}\pi_{a,t,r}=1,
  \quad \forall (a,t)\in\mathcal{A}_T,
  \label{alexa:row}\\
& \sum_{(a,t)\in\mathcal{A}_T}\pi_{a,t,r}=1,
  \quad \forall r\in\mathcal{K},
  \label{alexa:col}\\
& L_r-\lambda(a,t;P)\le M\,(1-\pi_{a,t,r}),
  \quad \forall (a,t),\,r,
  \label{alexa:link1}\\
& \lambda(a,t;P)-L_r\le M\,(1-\pi_{a,t,r}),
  \quad \forall (a,t),\,r,
  \label{alexa:link2}\\
& L_{r-1}\ge L_{r},
  \quad r=2,\ldots,N,
  \label{alexa:sort}\\
& L_j=\bar{\lambda}_j,
  \quad j=1,\ldots,k-1,
  \label{alexa:freeze}\\
& \pi_{a,t,r}\in\{0,1\},\ L_r\ge 0,
  \quad \forall (a,t),\,r.
  \label{alexa:bin}
\end{align}
\end{subequations}

In this formulation, equalities ~\eqref{alexa:row}--\eqref{alexa:col} allow to express $(\pi_{a,t,r})$ as a 
permutation matrix while \eqref{alexa:link1}--\eqref{alexa:link2} enforce
$L_r=\lambda(a,t;P)$ when $\pi_{a,t,r}=1$. Inequalities \eqref{alexa:sort} order the ranks 
non-increasingly, and \eqref{alexa:freeze} fixes the previously optimized ranks
to the constant values $\bar{\lambda}_j$. This formulation allows a single model to cover every
$k\in\mathcal{K}$. For instance, when $k=1$ the freezing range is empty, the objective is
$L_1=\lambda^{\downarrow}_{1}(P)$ and no constant from a previous level is
needed. Besides, the big-$M$ term in~\eqref{alexa:link1}--\eqref{alexa:link2} switches the two
inequalities on or off. If $\pi_{a,t,r}=1$, the right-hand side is $0$ and both
force $L_r=\lambda(a,t;P)$ while $\pi_{a,t,r}=0$ leads to a right-hand side value equal to $M$ which amounts to relax the constraints \eqref{alexa:link1}--\eqref{alexa:link2}. The latter case requires $M\ge\lambda(a,t;P)$ for all
$(a,t)\in\mathcal{A}_T$ and $P\in\mathcal{F}$, that is to say any smaller value would cut
feasible schemes from $\mathcal{F}$. 
In what follows, show how to derive a valid value for such a constant.

\begin{proposition}
\label{prop:alexa-bigm} 
The constant
\[
  M=\max_{(a,t)\in\mathcal{A}_T}\frac{\textnormal\maxSeg}{c(a)}
  \bigl(\max_{i\neq j} r(i,j,a,t)\bigr)\sum_{d\in D}\nu(d,t)
\]
satisfies $0\le\lambda(a,t;P)\le M$ for all $(a,t)\in\mathcal{A}_T$ and
$P\in\mathcal{F}$.
\end{proposition}
\begin{proof}
Non-negativity stems immediately from the load definition~\eqref{eq:load-P}. As for the
upper bound, inequality \eqref{ctn:segments} allows at most $\maxSeg$ non-zero terms per demand in the
inner sum of~\eqref{eq:load-P}, each bounded by
$(\max_{i\neq j}r(i,j,a,t))\,\nu(d,t)$; summing over $d\in D$ and dividing by
$c(a)$ gives $\lambda(a,t;P)\le M$.
\end{proof}
 
\begin{lemma}
\label{lem:alexa-sort}
Fix $P\in\mathcal{F}$ and consider $\mathcal{M}_k^{\mathrm{ALEXA}}$ without the
freezing constraints~\eqref{alexa:freeze}. Let $x$ be the routing variables
induced by $P$.
\begin{enumerate}
  \item[(i)] $P$ extends to a feasible point $(x,\pi,L)$ of the model.
  \item[(ii)] Every feasible point $(x,\pi,L)$ satisfies
    $L_r=\lambda^{\downarrow}_{r}(P)$ for all $r\in\mathcal{K}$; that is, $L$ is
    the non-increasing rearrangement of the loads of $P$.
\end{enumerate}
\end{lemma}

\begin{proof}
\emph{(i)} Let $\sigma:\mathcal{A}_T\to\mathcal{K}$ be a bijection ordering the
loads non-increasingly, i.e.\ $\lambda(\sigma^{-1}(1);P)\ge\cdots\ge
\lambda(\sigma^{-1}(N);P)$. Set $\pi_{a,t,r}=\mathbf{1}[r=\sigma(a,t)]$ and
$L_r=\lambda^{\downarrow}_{r}(P)$. Then~\eqref{alexa:row}--\eqref{alexa:col} hold
since $\pi$ is a permutation matrix, and~\eqref{alexa:sort} holds by the choice
of $\sigma$. For the linking constraints, when $\pi_{a,t,r}=1$ we have
$r=\sigma(a,t)$, so $L_r=\lambda(a,t;P)$
and~\eqref{alexa:link1}--\eqref{alexa:link2} hold with equality; when
$\pi_{a,t,r}=0$ they hold with slack, since $L_r,\lambda(a,t;P)\in[0,M]$
by~\cref{prop:alexa-bigm}. Hence $(x,\pi,L)$ is feasible.

\emph{(ii)} Let $(x,\pi,L)$ be any feasible point.
By~\eqref{alexa:row}--\eqref{alexa:col} and the integrality~\eqref{alexa:bin},
$\pi$ is a permutation matrix, defining a bijection
$\sigma:\mathcal{A}_T\to\mathcal{K}$ with $\pi_{a,t,\sigma(a,t)}=1$. For each
$(a,t)$, taking $r=\sigma(a,t)$ in~\eqref{alexa:link1}--\eqref{alexa:link2}
forces $L_{\sigma(a,t)}=\lambda(a,t;P)$. Thus $(L_r)_{r\in\mathcal{K}}$ is a
rearrangement of the loads of $P$, and it is non-increasing
by~\eqref{alexa:sort}; therefore $L_r=\lambda^{\downarrow}_{r}(P)$ for all $r$.
\end{proof}

\begin{theorem}
\label{thm:alexa-correctness}

For every $k\in\mathcal{K}$, $\mathcal{M}_k^{\mathrm{ALEXA}}$, namely~\eqref{mod:alexa}
satisfies~\eqref{eq:lex-recursion1}: its optimal value is
$\bar{\lambda}_k=\min_{P\in\mathcal{F}_{k-1}}\lambda^{\downarrow}_{k}(P)$, and
the scheme $P^\star_k$ read from any optimal solution belongs to $\mathcal{F}_k$.

\end{theorem}
\begin{proof}
By~\cref{lem:alexa-sort}, the feasible set of $\mathcal{M}_k^{\mathrm{ALEXA}}$
without freezing projects onto
$\{(P,\lambda^{\downarrow}(P)):P\in\mathcal{F}\}$. In particular
$L_j=\lambda^{\downarrow}_{j}(P)$ for every $j$, so~\eqref{alexa:freeze} holds if
and only if $\lambda^{\downarrow}_{j}(P)=\bar{\lambda}_j$ for all $j<k$, that is,
if and only if $P\in\mathcal{F}_{k-1}$ by~\eqref{eq:lex-recursion1}. On that
restricted set the objective $L_k$ equals $\lambda^{\downarrow}_{k}(P)$, hence
$\min L_k=\bar{\lambda}_k$. At an optimal solution we therefore have
$P^\star_k\in\mathcal{F}_{k-1}$ and
$\lambda^{\downarrow}_{k}(P^\star_k)=L_k=\bar{\lambda}_k$, which is precisely
$P^\star_k\in\mathcal{F}_k$.
\end{proof}


\subsection{STELA: Sequential Threshold-Exceedance Leximin Algorithm}\label{subsec:stela}

Instead of modeling all the possible permutations, the idea of STELA is to rely on threshold to compute each load rank. In other words, at iteration $k$, it computes the smallest threshold value
exceeded by at most $k-1$ loads, which equals $\lambda^{\downarrow}_{k}(P)$,
while preserving all previously optimized ranks. Now for an optimal routing scheme $P\in\mathcal{F}$ and a threshold value denoted $z\ge 0$, let $E_P(z)=\{(a,t)\in\mathcal{A}_T:\lambda(a,t;P)>z\}$.

\begin{proposition}
\label{prop:stela-threshold}
For every $P\in\mathcal{F}$ and $k\in\mathcal{K}$,
$\lambda^{\downarrow}_{k}(P)=\min\{z\ge 0:\lvert E_P(z)\rvert\le k-1\}$.
\end{proposition}
\begin{proof}
For a given threshold value $z\ge\lambda^{\downarrow}_{k}(P)$, only ranks $1,\ldots,k-1$ can exceed $z$,
so $\lvert E_P(z)\rvert\le k-1$. Now for $z<\lambda^{\downarrow}_{k}(P)$, the $k$
largest loads exceed $z$, so $\lvert E_P(z)\rvert\ge k$. Hence
$\lambda^{\downarrow}_{k}(P)$ is the smallest such threshold.
\end{proof}

Based on this idea, we give the following MILP formulation for the $T$-\srchallenge{} problem. We introduce, for each threshold level
$j\in\{2,\ldots,N\}$ and each arc--time pair $(a,t)\in\mathcal{A}_T$, a binary
variable $y_{a,t}^{(j)}\in\{0,1\}$ such that:

\[
  y_{a,t}^{(j)} =
  \begin{cases}
    1, & \text{if } \lambda(a,t;P) \text{ exceeds the level-$j$ threshold,}\\
    0, & \text{otherwise.}
  \end{cases}
\]
We introduce the following cardinality constraint $\sum_{(a,t)\in\mathcal{A}_T} y_{a,t}^{(j)} \leq j-1$, where the left-hand side express $\lvert E_P(z)\rvert$, and that allows to associate the level-$j$ threshold to $\lambda^{\downarrow}_{j}(P)$, as stated in~\cref{prop:stela-threshold}. For a
threshold value $z$, the exceedance condition is naturally written
\begin{equation}
\label{bilinear constraint}
  \lambda(a,t;P) \leq z + \bigl(\bar{\lambda}_1 - z\bigr)\,y_{a,t}^{(j)},
\end{equation}
which is linear when $z$ is a constant and represents a fixed level, $z=\bar{\lambda}_j$; but
becomes quadratic, through the product $z\,y_{a,t}^{(j)}$, when $z$ is itself a decision variable, as at the current level.

\paragraph{Level 1}
The first subproblem $\mathcal{M}_1^{\mathrm{STELA}}$ is written as:
\begin{equation}
  \bar{\lambda}_1=\min_{P\in\mathcal{F},\,z\ge0}
  \{z:\lambda(a,t;P)\le z,\ \forall(a,t)\in\mathcal{A}_T\},
  \label{eq:stela-l1}
\end{equation}
whose value $\bar{\lambda}_1$ upper-bounds every load at all later levels.

\paragraph{Level $k\ge 2$}
Given frozen values $\bar{\lambda}_1,\ldots,\bar{\lambda}_{k-1}$, the subproblem
$\mathcal{M}_k^{\mathrm{STELA}}$ is
\begin{subequations}
\small
\label{mod:stela}
\begin{align}
\min_{P\in\mathcal{F},\,z\ge 0,\,y}\ & z
  \label{stela:obj}\\
\text{s.t.}\ & \eqref{ctn:flow}\text{--}\eqref{ctn:budget_3},
  \label{stela:feas}\\
& \lambda(a,t;P)\le \bar{\lambda}_1,\quad \forall (a,t),
  \label{stela:gb}\\
& \lambda(a,t;P)\le \bar{\lambda}_j+(\bar{\lambda}_1-\bar{\lambda}_j)\,y^{(j)}_{a,t},
  \quad \forall (a,t),\ 2\le j<k,
  \label{stela:freeze-t}\\
& \sum_{(a,t)} y^{(j)}_{a,t}\le j-1,\quad 2\le j<k,
  \label{stela:freeze-c}\\
& \lambda(a,t;P)\le z+\bar{\lambda}_1\,y^{(k)}_{a,t},\quad \forall (a,t),
  \label{stela:cur-t}\\
& \sum_{(a,t)} y^{(k)}_{a,t}\le k-1,
  \label{stela:cur-c}\\
& 0\le z\le \bar{\lambda}_{k-1},
  \label{stela:mon}\\
& y^{(j)}_{a,t}\in\{0,1\},\quad \forall (a,t),\ 2\le j\le k.
  \label{stela:bin}
\end{align}
\end{subequations}

Constraints~\eqref{stela:freeze-t}--\eqref{stela:freeze-c} preserve the optimal
value of each preceding level ($\le j-1$ loads exceed $\bar{\lambda}_j$);
\eqref{stela:cur-t}--\eqref{stela:cur-c} encode the current level;
\eqref{stela:mon} enforces monotonicity of $(\bar{\lambda}_k)_k$.

\begin{remark}
\label{rem:linearization}
At a given level $k > $ 2, STELA avoids the quadratic term by replacing~\eqref{bilinear constraint}
with the pair~\eqref{stela:gb} and~\eqref{stela:cur-t},
\[
\lambda(a,t;P)\le \bar{\lambda}_1
\quad\text{and}\quad
\lambda(a,t;P)\le z+\bar{\lambda}_1\,y^{(k)}_{a,t}.
\]
Since $z\le\bar{\lambda}_{k-1}\le\bar{\lambda}_1$ by~\eqref{stela:mon}, these two
constraints together are equivalent to the quadratic term:
\begin{itemize}
  \item if $y^{(k)}_{a,t}=0$, the pair gives $\lambda(a,t;P)\le\min(z,\bar{\lambda}_1)=z$,
    matching the bilinear form evaluated at $y=0$;
  \item if $y^{(k)}_{a,t}=1$, the first gives $\lambda(a,t;P)\le\bar{\lambda}_1$,
    which coincides with the quadratic form evaluated at $y=1$.
\end{itemize}
Since $\bar{\lambda}_1$ is a constant fixed after the first iteration, every
subproblem $\mathcal{M}^{\mathrm{STELA}}_k$, $k>1$, is a standard mixed-integer
linear program.
\end{remark}

\begin{theorem}
\label{thm:stela-correctness}

For every $k\in\mathcal{K}$, $\mathcal{M}_k^{\mathrm{STELA}}$, namely~\eqref{eq:stela-l1} for $k=1$ and~\eqref{mod:stela} for $k\ge2$,
satisfies~\eqref{eq:lex-recursion1}: its optimal value is
$\bar{\lambda}_k=\min_{P\in\mathcal{F}_{k-1}}\lambda^{\downarrow}_{k}(P)$, and
the scheme $P^\star_k$ read from any optimal solution belongs to $\mathcal{F}_k$.
\end{theorem}

\begin{proof}
By induction on $k$. For $k=1$, the constraints of~\eqref{eq:stela-l1} force
$z\ge\max_{(a,t)}\lambda(a,t;P)=\lambda^{\downarrow}_{1}(P)$, so~\eqref{eq:stela-l1}
minimizes $\lambda^{\downarrow}_{1}$ over $\mathcal{F}_0=\mathcal{F}$: its value
is $\bar{\lambda}_1$ and any optimal $P^\star_1$ attains it, i.e.\
$P^\star_1\in\mathcal{F}_1$. Let now $k\ge2$ and assume the claim for all $j<k$;
write $v_k$ for the optimal value of $\mathcal{M}_k^{\mathrm{STELA}}$.

\emph{Upper bound: $v_k\le\bar{\lambda}_k$.} Fix a scheme
$P\in\mathcal{F}_{k-1}$ attaining $\bar{\lambda}_k$, and build a candidate
solution by setting $z=\bar{\lambda}_k$,
$y^{(j)}_{a,t}=\mathbf{1}\bigl[\lambda(a,t;P)>\bar{\lambda}_j\bigr]$
$(2\le j<k)$ and $y^{(k)}_{a,t}=\mathbf{1}\bigl[\lambda(a,t;P)>z\bigr]$.
Since $P\in\mathcal{F}_{k-1}$, we have $\lambda^{\downarrow}_{1}(P)=\bar{\lambda}_1$,
so every load is at most $\bar{\lambda}_1$ and the global bound~\eqref{stela:gb}
holds. For each $j$ with $2\le j<k$, again $\lambda^{\downarrow}_{j}(P)=\bar{\lambda}_j$,
so by~\cref{prop:stela-threshold} at most $j-1$ loads exceed $\bar{\lambda}_j$;
with the above choice of $y^{(j)}$ this gives both the threshold
constraint~\eqref{stela:freeze-t} and the cardinality
constraint~\eqref{stela:freeze-c}. Likewise
$z=\bar{\lambda}_k=\lambda^{\downarrow}_{k}(P)$ is exceeded by at most $k-1$
loads, yielding~\eqref{stela:cur-t}--\eqref{stela:cur-c}; and
$z=\bar{\lambda}_k\le\bar{\lambda}_{k-1}$ gives~\eqref{stela:mon}. This feasible
point has objective value $\bar{\lambda}_k$, hence $v_k\le\bar{\lambda}_k$.

\emph{Lower bound: $v_k\ge\bar{\lambda}_k$.} Let $(P_k^\star,z^\star,y)$ be an
optimal solution. We first show $P_k^\star\in\mathcal{F}_{k-1}$ by an inner
induction on $j<k$. For $j=1$, the global bound~\eqref{stela:gb} gives
$\lambda^{\downarrow}_{1}(P_k^\star)\le\bar{\lambda}_1$, and since
$P_k^\star\in\mathcal{F}$ the minimality of $\bar{\lambda}_1$ forces
$\lambda^{\downarrow}_{1}(P_k^\star)=\bar{\lambda}_1$.
For $2\le j<k$, the threshold constraints~\eqref{stela:freeze-t} and the
cardinality constraints~\eqref{stela:freeze-c} imply that at most $j-1$ loads
exceed $\bar{\lambda}_j$, so $\lambda^{\downarrow}_{j}(P_k^\star)\le\bar{\lambda}_j$
by~\cref{prop:stela-threshold}; combined with $P_k^\star\in\mathcal{F}_{j-1}$
(inner hypothesis) and the minimality of $\bar{\lambda}_j$, this gives
$\lambda^{\downarrow}_{j}(P_k^\star)=\bar{\lambda}_j$. Hence
$P_k^\star\in\mathcal{F}_{k-1}$. Finally, the current-level
constraints~\eqref{stela:cur-t}--\eqref{stela:cur-c}
and~\cref{prop:stela-threshold} give
$\lambda^{\downarrow}_{k}(P_k^\star)\le z^\star$, so
\[
  \bar{\lambda}_k
  =\min_{P\in\mathcal{F}_{k-1}}\lambda^{\downarrow}_{k}(P)
  \le\lambda^{\downarrow}_{k}(P_k^\star)\le z^\star=v_k.
\]

Combining both bounds, we obtain $v_k=\bar{\lambda}_k$ which completes the proof and yields
$\lambda^{\downarrow}_{k}(P_k^\star)=\bar{\lambda}_k$ and
$P_k^\star\in\mathcal{F}_k$.
\end{proof}

By~\cref{thm:alexa-correctness,thm:stela-correctness}, both ALEXA and STELA formulations
meet~\eqref{eq:lex-recursion1} and return a lexicographically optimal
routing scheme when driven by~\cref{alg:lex-template}.

\section{Numerical Experiments}
\label{sec:experiments}

\subsection{Benchmark instance generation procedure}
\label{datageneration}
We developed benchmark instances for the $T$-ASR problem. Each instance
combines a network topology, a time-dependent traffic matrix, a set of intervention
scenarios, and operational parameters. The traffic matrices are constructed and then selected based on a difficulty indicator, which is derived from the gap between the lower bound given by the multi-commodity flow problem and an upper bound for the segment-routing obtained thanks to a greedy algorithm. The reconfiguration budget is set at a level lower than the cost of performing independent optimizations for each period. This approach is chosen to preserve the temporal coupling (\textit{i.e. } the interdependence between different time periods). All the instance files are provided in the reference hereafter~\cite{roadef2026_instances}-~\cite{stela_alexa_code}. 

\subsection{Experimental setup}
\label{subsec:setup}

In this section, we compare ALEXA and STELA on the so-called \texttt{setAP}~\cite{stela_alexa_code}, one of the
benchmark families developed for the $T$-ASR problem. This set contains $125$
instances grouped by $N\in\{100,200,300,400,500\}$. 
This set of instance was generated with a number of nodes varying from 10 to 50, by increments of 10. The number of arcs is five times the number of nodes. For each number of nodes, the number of demands varies from 10 to 50, by increments of 10.
Both algorithms follow the framework described in \cref{alg:lex-template} and
differ only in the formulation $\mathcal{M}_k$ solved at each step, so the comparison
isolates the effect of the encoding. We set a CPU time limit of 20 minutes for both approaches and all the instances tested. When the time limit is reached, the run is stopped and we record
the largest rank $k$ reached, together with the values
$\bar{\lambda}_1,\ldots,\bar{\lambda}_k$. Due to the sequential nature of the
lexicographic procedure, reaching rank $k$ means that the first $k$ ranks have
been solved to optimality. Since both approaches solve one lexicographic level
through one MILP subproblem, $k$ measures the depth reached within the
computational time limit. \cref{alg:lex-template} along with both formulations is implemented in \texttt{Python~3.10.11} using \texttt{networktools}~\cite{networktools} library and \texttt{Gurobi~13.0.2} as a MILP solver, under deterministic settings on an \texttt{Intel Xeon~2.80\,GHz} machine with 32\,GB RAM and \texttt {Ubuntu~24.04}.
The complete implementation of both algorithms are publicly available~\cite{stela_alexa_code}.


\subsection{Numerical tolerances}
\cref{alg:lex-template} assumes exact MILP optimization and exact equality
between successive levels.
In the computational
experiments, these assumptions are relaxed: each MILP subproblem is solved with
a relative MIP gap $\gamma$, and previously optimized levels are enforced within
a numerical tolerance $\varepsilon$. We use $\gamma=10^{-8}$ and
$\varepsilon=10^{-9}$ for both algorithms. This precision is particularly relevant for lexicographic
optimization, since the subproblems are solved sequentially: the tolerance used
at one level affects the feasible region considered at subsequent levels and may
lead to different approximate outcomes at later ranks. The stability of
approximate lexicographic optimization has been studied in~\cite{abernethy2024}.

\subsection{Ranks solved within the time limit}
\label{subsec:results-depth}
\begin{figure}[h!]
    \centering
    \includegraphics[width=0.6\linewidth]{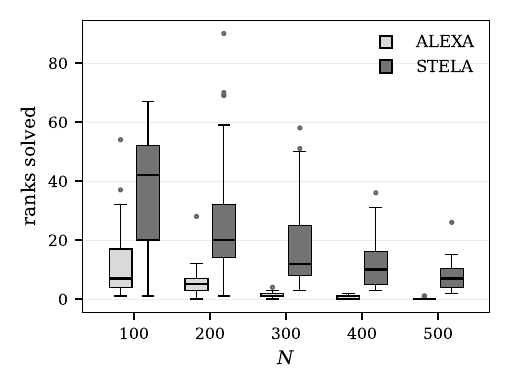}
    \caption{\small Number of solved ranks of the sorted load vector within the
20-minute time limit, for each instance class. Boxes show the main range of
the data, and individual points represent outliers.}
    \label{fig:boxplot_stela_vs_alexa}
\end{figure}

\cref{fig:boxplot_stela_vs_alexa} shows the number of lexicographic ranks
solved by each algorithm within the time limit. Observe that STELA performs better
than ALEXA on almost all instances. In particular, it solves more ranks on $109$ out of
$115$ instances, reaches the same depth on $5$ instances, and solves fewer
ranks on only one instance.

This behaviour follows from the structure of the per-rank formulations
$\mathcal{M}_k^{\mathrm{ALEXA}}$ and $\mathcal{M}_k^{\mathrm{STELA}}$. Indeed, fixing a lexicographic
level in ALEXA constrains the continuous load variables $L_j$ but leaves the
$N\times N$ permutation matrix unconstrained, thereby providing greater flexibility in the solutions. 
Conversly, although both formulations have quadratic worst-case size at the final ranks, STELA contains $N(k-1)$ threshold-exceedance binary variables and is substantially smaller for early and intermediate ranks $k\ll N$. Thus,
the additional computational effort per level increases more slowly for the latter algorithm.

\section{Concluding remarks}
\label{sec:conclusion}
In this paper, we introduced $T$-\srchallenge{} problem, a multi-period segment-routing problem where
the sorted link load vector is minimized lexicographically under specific routing protocol and budget constraints. We proposed a decomposition of the problem into a sequence of MILP subproblems solving each rank, together with two different formulations for the rank subproblem: ALEXA and STELA.

Experiments on the \texttt{setAP} benchmark show that STELA performs better than
ALEXA on almost all the tested instances and within the fixed CPU time limit. These results make STELA a more
promising approach to tackle larger instances. Future work will investigate more
scalable solution methods, such as column generation and branch-and-price,
while preserving the lexicographic framework.


\bibliographystyle{unsrt}
\bibliography{main}

\end{document}